\scrollmode
\documentclass[10pt]{amsart}
\usepackage{amssymb}
\usepackage[all,arc]{xy}
\usepackage{multicol}

\theoremstyle{plain}
\newtheorem{prop}{Proposition}
\newtheorem{thm}[prop]{Theorem}
\newtheorem*{thmA}{Theorem A}
\newtheorem*{thmB}{Theorem B}
\newtheorem*{thmC}{Theorem C}
\newtheorem*{thmD}{Corollary D}

\newtheorem{cor}[prop]{Corollary}
\newtheorem{Lemma}[prop]{Lemma}

\theoremstyle{definition}
\newtheorem*{defi}{Definition}

\theoremstyle{remark}
\newtheorem{rem}[prop]{Remark}
\newtheorem{example}{Example}

\numberwithin{prop}{section}
\numberwithin{example}{section} 
\numberwithin{equation}{section}

\newcommand{\Z}{\mathbb{Z}}
\newcommand{\Q}{\mathbb{Q}}
\newcommand{\N}{\mathbb{N}}
\newcommand{\F}{\mathbb{F}}

\newcommand{\de}{\mathrm{def}}
\newcommand{\rk}{\mathrm{rk}}

\newcommand{\ab}{\mathrm{ab}}
\newcommand{\HNN}{\mathrm{HNN}}
\newcommand{\het}{\mathrm{ht}}

\newcommand{\caT}{\mathcal{T}}
\newcommand{\gG}{\mathcal{G}}

\newcommand{\image}{\mathrm{im}}

\newcommand{\argu}{\hbox to 1.5ex{\hrulefill}}  

\begin{document}
	\title{	homological approximations  for   Profinite   and Pro-$p$ limit groups }
\author{Jhoel S. Gutierrez}
\address{ 
		(Jhoel S. Gutierrez) Universidade Federal de Mato Grosso do Sul, Departamento de Matematica, Campus Nova Andradina,  Av. Reitor Per\'o, Nova Andradina - MS, 79750-000,  Brazil }
	\email{jhoel.sandoval@ufms.br}
	\thanks{The author was supported by CNPq-Brazil}
	\date{\today}
	
	\begin{abstract} 
	 We study homological approximations of the profinite completion of a limit group (see Thm.~A) and  obtain     the analogous of Bridson and  Howie's Theorem for the profinite completion of a  non-abelian limit group  (see Thm.~B).  
	\end{abstract}
	\maketitle 
	\section{Introduction}
	\label{s:intro}
		 Recently, M. Bridson and D. Kochloukova  (see \cite{Brid}) studied asymptotic behavior of the dimensions of the
	homology groups of normal subgroups of finite index   of  a limit group  that allowed them to calculate  its analytic Betti numbers.

	In the present paper we study asymptotic behavior of the dimensions of the
	homology groups of open normal subgroups of the  profinite completion  of  a limit group. The profinite completion of a limit group belongs to the class of pro-$\mathcal{C}$  groups $\mathcal{Z(C)}$ studied by P.~Zalesskii and T.~Zapata   (see Section 3 in \cite{TZ} for more details).  In our first result  we generalize  Theorem F(c) (see \cite{TZ}) for the profinite completion of limit groups. 
	
	\begin{thmA}\label{thm:A}
		Let  $G$ be a limit group,  $p$ a prime number, and   $\{U_{i}\}_{i\geq 1}$  a sequence of open  normal subgroups of $\widehat{G}$  such
		that  $U_{i+1}\leq U_i$ for all $i\geq 1$ and $cd_p(\displaystyle\bigcap_{i\geq 1}{U_i})\leq 2$. Then
		\begin{itemize}
			\item[(1)]
			$\displaystyle\lim_{\overrightarrow{i}}{\dim_{\F_p}{ H_j(U_i,\F_p)}/[\widehat{G}:U_i]}=0$ for $j\geq 3$;
			\item[(2)]If $\{[\widehat{G}:U_i]\}_{i\geq 1}$ tends to infinity we have
			$$\lim_{\overrightarrow{i}}{\dim_{\F_p}{\big(H_1(U_i,\F_p)-H_2(U_i,\F_p)\big)}/[G:U_i]}=-\chi_p(\widehat{G});$$
			\item[(3)] If   $\displaystyle\bigcap_{i\geq 1}{U_i}=1$ we have 
			$\displaystyle\lim_{\overrightarrow{i}}\dim_{\F_p}{{H_2(U_i,\F_p)}/[\widehat{G}:U_i]}=0$\,\, \\ and\,\,  $\displaystyle\lim_{\overrightarrow{i}}{\dim_{\F_p}{H_1(U_i,\F_p)}/[\widehat{G}:U_i]}=-\chi_p(\widehat{G})\geq 0$.  In particular the rank gradient  of $\widehat{G}$ is non-negative.
		\end{itemize}
	\end{thmA}
	
	Note that a finitely generated  non-trivial normal  subgroup of a free group as well as of its profinite completion is  of finite index (see \cite[\S 8.6, Thm.~8.6.5]{ZR}). 
	Celebrated   Bridson and  Howie's Theorem (see \cite[ Thm~3.1]{Bri} ) states that a finitely generated non-trivial normal subgroup   of a limit group is of finite index. Using methods of M.~Shusterman (see \cite{shu2}) we prove the following analog of   Bridson and  Howie's Theorem.
	
	\begin{thmB}
		Let $G$ be a  non-abelian limit group, and  let $N\neq 1$ be a finitely generated  normal closed subgroup  of  $\widehat{G}$.  Then $[\widehat{G}:N]<\infty$. 
	\end{thmB}

	Theorem A was proved by Kochloukova and Zalesskii (see \cite{AC}) for the pro-$p$ analogs of limit groups introduced by  D. Kochloukova and P.~Zalesskii (see \cite{AD}) using the construction of extension of centralizers.  
	The following result establishes   it over $\Q_p$.  	 
	\begin{thmC} 
		Let $G$ be a pro-$p$ limit group and let $\{U_{i}\}_{i \geq 1}$ be  a sequence of open subgroups of $G$ such that $U_{i+1}\leq U_i$  for all $i\geq 1$ and $cd_p(\displaystyle\bigcap_{i\geq 1}{U_i})\leq 2$. Then
		\begin{itemize}
			\item[(1)]
			$\displaystyle\lim_{\overrightarrow{i}}{\dim_{\Q_p}{\big(\Q_p\otimes_{\Z_p} H_j(U_i,\Z_p)\big)}/[G:U_i]}=0$  for $j\geq 3$;
			\item[(2)] If $\{[G:U_i]\}_{i\in I}$ tends to infinity we have
			$$\lim_{\overrightarrow{i}}{\big(\dim_{\Q_p}{(\Q_p\otimes_{\Z_p}H_1(U,\Z_p))-\dim_{\Q_p} (\Q_p\otimes_{\Z_p}H_2(U,\Z_p))\big)}/[G:U_i]}=-\chi_p(G);$$ 
			\item[(3)]If  $\displaystyle\bigcap_{i\geq 1}{U_i}=1$ we have \\
			$$\displaystyle\lim_{\overrightarrow{i}}\dim_{\Q_p}{\big(\Q_p\otimes_{\Z_p}{H_2(U_i,\Z_p)\big)}/[G:U_i]}=0$$  and\\ $$\displaystyle\lim_{\overrightarrow{i}}{\dim_{\Q_p}{\big(\Q_p\otimes_{\Z_p}H_1(U_i,\Z_p)\big)}/[G:U_i]}=\displaystyle\lim_{\overrightarrow{i}}{\rk_{\Q_p}(U_i)}/[G:U_i]=-\chi_p(G)$$.
		\end{itemize}
	\end{thmC}
	
	\begin{thmD}
		Let $G$ be a pro-$p$ limit group and let  $\{U_{i}\}_{i\geq 1}$ be a sequence of open subgroups of $G$ such that    $U_{i+1}\leq U_i$ and  $\displaystyle\bigcap_{i\geq 1}{U_i}=1$.  Then  $$\displaystyle\varinjlim_{i}{T(U_i)}/[G:U_i]=0,$$ where $T(U_i)$ is the rank of the torsion group of $U_i^{\ab}.$
	\end{thmD}
	\noindent
	{\bf Acknowledgement:}
	The author would like to thank P.~Zalesskii for a very useful comment concerning an
	earlier version of the paper. 
	
	\section{Preliminaries}
	This section contains certain preliminary results which will be of use later. 
	\subsection{Limit groups}

\subsubsection{Extension of centralizers}
	\label{ss:et}
	Starting from a limit group $G$ there is a standard procedure
	to construct a limit group $G(C,m)$, where $C\subseteq G$ is a maximal cyclic subgroup
	of $G$ and $m\in\N$. This procedure is known as {\em extension of centralizers}, i.e., if $G$ is a limit group, then
	\begin{equation}
	\label{eq:excen}
	G(C,m)=G\star_C(C\times\Z^m)
	\end{equation}
	is again a limit group (see \cite[Lemma 2 and Thm.~4]{KM2}).
	For short we call a limit group $G$ to be an {\em iterated extension of centralizers of a free group} (= i.e.c. free group),
	if there exists a sequence of limit groups $(G_k)_{0\leq k\leq n}$ such that
	\begin{itemize}
		\item[(E$_1$)] $G_0=F$ is a finitely generated free group, and $G_n\simeq G$;
		\item[(E$_2$)] for $k\in\{0,\ldots,n-1\}$ there exists a maximal cyclic subgroup $C_k\subseteq G_k$ and
		$m_k\in\N$ such that $G_{k+1}\simeq G_k(C_k,m_k)$.  
	\end{itemize}
	If $G$ is an i.e.c. free group, one calls the minimum number $n\in\N_0$ for which there exists
	a sequence of limit groups $(G_k)_{0\leq k\leq n}$ satisfying (E$_1$) and (E$_2$) the {\em level} of $G$.
	This number will be denoted it by $\ell(G)$. E.g., a finitely generated free group is an i.e.c. free group of level $0$, and a finitely generated free abelian group is an
	i.e.c. free group of level $1$.
	
	\subsubsection{The height of a limit group}
	\label{ss:hgt}
	By the second embedding theorem (see \cite[\S 2.3, Thm.~2]{KMRS}, \cite[Thm.~4]{KM2}), every limit group $G$ is isomorphic to a subgroup
	of an i.e.c. free group $H$. The {\em height} $\het(G)$ of $G$ is defined as
	\begin{equation}
	\label{eq:defht}
	\het(G)=\min\{\,\ell(H)\mid G\subseteq H,\ \text{$H$ an i.e.c. free group}\,\}.
	\end{equation}
	E.g., a limit group $G$ is of height $0$ if, and only if, it is a free group of finite rank,
	and non-cyclic finitely generated free abelian groups are of height $1$.
	
	\subsubsection{Limit groups as fundamental groups of graph of groups} The next proposition is well-known 
	(see for example \cite[Proposition~2.1]{jh})
	\begin{prop}
		\label{prop:prin}
		Let $G$ be a limit group of height $n\geq 1$. Then $G$ is isomorphic to the fundamental
		group $\pi_1(\gG^\prime,\Lambda_0,\caT_0)$ of a graph of groups $\gG^\prime$
		satisfying
		\begin{itemize}
			\item[(i)] $\Lambda_0$ is finite;
			\item[(ii)] for all $v\in V(\Lambda_0)$, $\gG^\prime_v$ is finitely generated abelian
			or a limit group of height at most $n-1$;
			\item[(iii)] for all $e\in E(\Lambda_0)$, $\gG^\prime_e$ is infinite cyclic or trivial.
		\end{itemize}
	\end{prop}


	The following property have  been shown by
	D.~Kochloukova in \cite{D}. 
	
	\begin{Lemma}
		\label{l4}
		For a limit group $G$ its Euler characteristic $\chi(G)$ is non-positive. Moreover, 
		$\chi(G)=0$ if, and only if,  $G$ is abelian. 
	\end{Lemma}
	By construction limit groups are of type $FP_{\infty}$ over $\Z$  (and so over $\Q$) and of
	finite cohomological dimension (see \cite[ Cor.~8]{D}))
	\subsection{The $p$-deficiency} 
	The  notion of $p$-\textbf{deficiency}  ($p$ a prime) of a profinite group $G$  was introduced in \cite[\S 2.2]{G}. Let $G$ be a finitely generated profinite group. Denote by $d(G)$ its minimal
	number of generators. If $M$ is a non-zero finite $G$-module we denote by $\dim M$
	the length of M as $\Z$-module and put 
	\begin{align}\label{eq01}
		\overline{\chi_2}(G,M)=\frac{-\dim H^2(G,M)+\dim  H^1(G,M)-\dim   H^0(G,M)}{\dim(M) }\in \Q\cup \{+\infty\}
	\end{align}
	and 
	\begin{align}
		\overline{\chi_1}(G,M)=\frac{\dim H^1(G,M)-\dim H^0(G,M)}{\dim(M)} \in \Q\cup \{-\infty\}.
	\end{align}
	\begin{defi}
		If $N$ is a closed subgroup of $G$ and $\mathcal{M}_p(N)$  is the set of all finite $\Z_p[[G]]$-modules on which $N$ acts trivially, we define  the following invariant:
		
		\begin{align}	
			\de_p(G,N)=\displaystyle\inf_{M\in \mathcal{M}_p(N) }\{1+\overline{\chi_2}(G,M)\}.
		\end{align}
		For simplicity, we put
		\begin{align}
			\de_p(G):=\de_p(G,\{1\}) 
		\end{align}
		The number $\de_p(G)$ is called the $p$-deficiency of G.
	\end{defi}
	
	Comparing this invariant
	with the deficiency  $\de(G)$ of $G$  we observe that
	\begin{align}\label{eq001}
		\de(G)\leq \de_p(G)\leq \dim_{\F_p}H^1(G, \F_p)-\dim_{\F_p}H^2(G, \F_p).
	\end{align}
	
	In \cite[Proposition 3.2]{G} the following was shown. 
	\begin{prop} \label{pp1}
		Let $G$ be a finitely generated profinite group and let $N$ be a
		normal subgroup of infinite index such that $H_1(N, \F_p)\neq 0$. If $\de_p(G, N)\geq 2$, then $\dim_{\F_p}H_1(N, \F_p)$ is infinite. 
	\end{prop}

	\section{Homological approximations for the   profinite completion  of  limit group}
	\label{nsb}
\begin{defi}[Euler characteristic]
		Let   $G$ be a profinite group  and $p$ a prime number   satisfying the following  conditions:   
		\begin{itemize}
			\item[(i)]  The  $p$-cohomological dimension $cd_p(G)$ of $G$  is finite;
			\item[(ii)] $\dim_{\F_p}H_n(G,\F_p)< \infty$, for all $n\in \N_0$.
		\end{itemize} 
	The $p$-Euler characteristic of    $G$  is defined as
	\begin{align} \label{e01}
 \chi_p(G)=\displaystyle\sum_{0\leq i\leq cd_p(G)}{(-1)^i \dim_{\F_p}{H_i(G,\F_p)}}
	\end{align} 
\end{defi}
The following Theorem  has  been shown by  T.~Zapata and P.~Zalesskii.
\begin{thm}[Theorem F in  \cite{TZ}]\label{TZa}
Let $G$ be  a limit group. Then:
\begin{itemize}
\item[(a)] $\widehat{G}$ is of homological type $FP_{\infty}$ over $\F_p$ and over $\Z_p$. 
	\item[(b)] The  Euler  characteristic $\chi({G})$ of $G$  coincides 	with the  p-Euler characteristic $\chi_p(\widehat{G})$ of $\widehat{G}$; hence 	it is always non-negative  and equals $0$ if, and only if, $\widehat{G}$ is abelian.
	\item[(c)] If $U_1\geq U_2 \geq U_3 \geq \cdots$ is a descending sequence of open normal subgroups 	of $\widehat{G}$ with trivial intersection, then 
	\begin{equation*}\label{e02}	
	\displaystyle\lim_{i\rightarrow\infty} \frac{\dim_{\F_p}H_j(U_i, \F_p)}{[\widehat{G} : U_i]}=\displaystyle\lim_{i\rightarrow\infty} \frac{\dim_{\F_p}H^j(U_i, \F_p)}{[\widehat{G} : U_i]}= \left \{ \begin{matrix} -\chi_p(\widehat{G}), & \mbox{if }\mbox{ j=1}
	\\ 0, & \mbox{otherwise }\end{matrix}\right. 
	\end{equation*}
	
\end{itemize}
\end{thm}
\begin{rem}
The proof of the following Lemma is analogous to the proof of its pro-p version (see  \cite[\S 5,  Thm.~5.1]{AC}). Note that,  the functions $\rho_1$ and $\rho_2$  do not depend on the domain $G$  when the sequence $\{U_i\}_{i\geq 1}$  is formed by normal open subgroups. 
\end{rem}

\begin{Lemma}\label{l20}
		Let $p$ be a prime number. Let G be a profinite group acting on a profinite tree
		$T$ such that $T/G$ is finite and all vertex and edge stabilizers are of type $FP_{\infty}$ over $\Z_p$.  Let  $j\geq 1$ be  an integer     and $\{U_i\}_{i\geq 1}$  a sequence of open  normal subgroups of $G$  such that for all $i$ we have $U_{i+1}\leq U_i$  and  
		\begin{eqnarray}
		\displaystyle\lim_{i\rightarrow\infty}\displaystyle\dim_{\F_p} H_j(U_i\cap G_v,\F_p)/[G_v:(U_i\cap G_v)]=\rho_1(v)<\infty\label{111}\\
		\displaystyle\lim_{i\rightarrow\infty}\displaystyle\dim_{\F_p} H_{j-1}(U_i\cap G_e,\F_p)/[G_e:(U_i\cap G_e)]=\rho_2(e)<\infty\label{2}
		\end{eqnarray}
		
		for all $v\in V(T)/G$ and $e\in E(T)/G$, where   
		$\rho_1(v)$, $\rho_2(e)$ are continuous functions with domains $\{v\}_{v\in V(T)/G}$ and $\{e\}_{e\in E(T)/G}$ respectively. Then 
		
		$$	\displaystyle\sup_{\overrightarrow{i}	}\displaystyle\dim_{\F_p} H_{j}(U_i,\F_p)/[G:U_i]\leq \displaystyle
		\sum_{v\in V(T)/G} \displaystyle\rho_1(v)+ \displaystyle
		\sum_{e\in E(T)/G} \displaystyle\rho_2(e)$$
		
		In particular, if  $\rho_1(v)$ and $\rho_2(e)$ are the zero map, then 
		$$	\displaystyle\lim_{i \rightarrow \infty
		}\displaystyle\dim_{\F_p} H_{j}(U_i,\F_p)/[G:U_i]=0$$
	\end{Lemma} 
	
	The Lemma \ref{l20}  will  turn out to be useful to generalize  the Theorem \ref{TZa}(c). 
\begin{proof}[\textbf{Proof of Theorem A}]

\noindent
		\\ \item[\textbf{(1).}]
		
		We induct on the height of G. First if height of $G$ is 0, then $G$ is a 
		free group, hence  $cd_p(\widehat{G} )=cd(G)=1$ and  $cd_p(U_i)= 1$. It follows that  $H_j(U_i,\F_p)=0$, for $j\geq 2$ and $i\geq 1$. 
		
		Assume that the theorem holds for limit groups  of smaller height.  Let
		$n$ be the height of G and $G_n=G_{n-1}*_{C_{n-1}}A_{n-1}$, where $A_{n-1}=C_{n-1}\times B$ is a free abelian of finite rank. By Proposition~\ref{prop:prin},
		$G$ is isomorphic to the fundamental group $\pi_1(\Upsilon,\Lambda,\caT)$
		of a graph of groups $\Upsilon$ based on a finite connected graph $\Lambda$ whose edge groups are either infinite cyclic or trivial, and whose vertex groups are either limit group of height at most $n-1$ or free abelian groups.

		Since the profinite topology on $G$ is efficient (see  \cite[\S 3, Thm.~3.8]{gru} ), then $\widehat{G}$ is isomorphic to the fundamental group of a finite graph of groups whose groups of edges are $\{\widehat{G}_e\}_{e\in E(\Lambda)}$, where $\widehat{G}_e\cong \widehat{\Z}$ or \{1\} and the vertex groups are
		$\{\widehat{G}_v\}_{v\in V(\Lambda)}$. 	By Theorem \ref{TZa}(a), $\widehat{G}_v$ and  $\widehat{G}_e$ are of homological type $FP_{\infty}$  over $\Z_p$. 
		
		For $v\in V(\Lambda)$ suppose that $G_v$ is non-abelian. Note that the height of $G_v$ is smaller than the height of $G$ since $G_v$ is inside of 	a conjugate of $G_{n-1}$ or a conjugate of $A_{n-1}$. By induction applied for the group
		$G_v$ 	for a fixed $j>2$ we have

		\begin{align}
			\rho_1(i,v)=\dim_{\F_p}{H_j(U_i\cap \widehat{G}_v,\F_p)}/[\widehat{G}_v:U_i\cap \widehat{G}_v]
		\end{align}
		tends to 0 as $i$ goes to infinity.

		If $G_v$  is abelian, then $\widehat{G}_v$ is free abelian profinite (isomorphic to $\widehat{\Z}^m$, where  $m$  is a finite natural number), hence $U_i\cap \widehat{G}_v\simeq \widehat{G}_v$  and  $$\dim_{\F_p}{H_j(U_i\cap \widehat{G}_v,\F_p)}= \dim_{\F_p}{H_j(U_{i+1}\cap\widehat{G}_v,\F_p)},$$ for every  $i\geq 1$. If  $\{[\widehat{G}_v:U_i\cap \widehat{G}_v]\}_{i\geq 1}$ tends to infinity, so for a fixed   $j\geq 0$ 
		\begin{align*}
			\rho_1(i,v)=\dim_{\F_p}{H_j(U_i\cap \widehat{G}_v,\F_p)}/[\widehat{G}_v:U_i\cap \widehat{G}_v]
		\end{align*}
		tends to 0 as i goes to infinity.  	If  $\{[\widehat{G}_v:U_i\cap \widehat{G}_v]\}_{i\geq 1}$ does not tend to
		infinity, then there is $i_0\in \N$ such that  $U_{i_0}\cap \widehat{G}_v=U_{i_0+1}\cap \widehat{G}_v=U_{i_0+2}\cap \widehat{G}_v=\cdots$, hence   $\displaystyle\bigcap_{i\geq i_0}{(U_i\cap \widehat{G}_v)}=U_{i_0}\cap \widehat{G}_v$.  So for $i\geq i_0$ we have 
		\begin{align}
			cd_p(U_{i}\cap \widehat{G}_v)= cd_p(U_{i_0}\cap \widehat{G}_v)=cd_p\big(\displaystyle\bigcap_{i\geq i_0}{(U_i\cap \widehat{G}_v)}\big)\leq cd_p\big(\displaystyle\bigcap_{i\geq 1}{U_i}\big)\leq 2, 
		\end{align}
		It follows  	$H_j(U_i\cap \widehat{G}_v,\F_p)=0$, para $i\geq i_0$,  $j\geq 3$. Therefore, for a fixed $j>2$,  	
		$\rho_1(i,v)$
		tends to 0 as i goes to infinity.
		
		Observe that $\widehat{G}_e$ is isomorphic to \{1\} or     $\widehat{\Z}$   for all $ e\in E(\Lambda)$, hence $cd_p(U_i\cap \widehat{G}_e)\leq 1$ and $H_{j-1}(U_i\cap \widehat{G}_e,\F_p)=0$, para $j\geq 3$.   Therefore,  we
		have for any $j>2$ that  
		\begin{align}
			\rho_1(i,e)=\dim_{\F_p}{H_j(U_i\cap \widehat{G}_e,\F_p)}/[\widehat{G}_e:U_i\cap \widehat{G}_e]
		\end{align}
		tends to 0 as i goes to infinity.

		Then by Lemma \ref{l20} the result follows for $j>2$.

		\item[\textbf{(2).}]  
		Since
		\begin{align}
			\chi_p(\widehat{G})=\chi_p(U_i)/[\widehat{G}:U_i]=\sum_{j\geq 0}{(-1)^j \dim_{\F_p}{H_j(U_i,\F_p)}/[\widehat{G}:U_i]}
		\end{align}
		We define  a map   $L(U):=\dim_{\F_p}H_1(U,\F_p)-\dim_{\F_p}H_2(U,\F_p)$, where $U\leq_o\widehat{G}$, hence we have    
		
		\begin{align}
			\chi_p(\widehat{G})=\sum_{j\geq 3}{ \dim_{\F_p}{H_j(U_i,\F_p)}/[\widehat{G}:U_i]}-{L(U_i)}/[\widehat{G}:U_i]+1/[\widehat{G}:U_i].
		\end{align}
		
		It follows  \begin{align}
			\chi_p(\widehat{G})=\sum_{j\geq 3}{\lim_{i\rightarrow\infty}{ \dim_{\F_p}H_j(U_i,\F_p)}/[\widehat{G}:U_i]}-\lim_{i\rightarrow\infty}{L(U_i)/[\widehat{G}:U_i]}+\lim_{i\rightarrow\infty}{1/[\widehat{G}:U_i]}.
		\end{align}
		
		Hence, 
		by applying $(1)$ and hypothesis, one concludes that
		$$\chi_p(\widehat{G})=-\lim_{\overrightarrow{i}}{L(U_i)/[\widehat{G}:U_i]}.$$
		
		\item[\textbf{(3).}] Note that if  $\displaystyle\bigcap_{i\geq 1}{U_i}=1$, then   ${[\widehat{G}:U_i]}_{i\geq 1}$ tends to $\infty$ as i goes to infinity. 
		
		For $v\in V(\Lambda)$ suppose that $G_v$ is non-abelian. Note that the height of $G_v$ is smaller than the height of $G$ since $G_v$ is inside of 	a conjugate of $G_{n-1}$ or a conjugate of $A_{n-1}$. By induction applied for the group
		$G_v$  we have 
		$$\displaystyle\lim_{\overrightarrow{i}}{\dim_{\F_p}{H_2(U_i\cap \widehat{G}_v,\F_p)}/[\widehat{G}_v:U_i\cap \widehat{G}_v]}=0,$$
		
		In the case that  $G_v$ is abelian we have that the same argument which was used 
		in order to prove  (1) implies that  $$\displaystyle\lim_{\overrightarrow{i}}{\dim_{\F_p}{H_2(U_i\cap \widehat{G}_v ,\F_p)}/[\widehat{G}:U_i]}=0.$$
		Observe that $U_i\cap \widehat{G}_e$ is isomorphic to \{1\} or     $\widehat{\Z}$   for all $ e\in E(\Lambda)$, hence
		we have  $\dim_{\F_p}H_1(U_i\cap \widehat{G}_e,\F_p)\leq 1$ for all $i\geq 1$. Therefore  $$\displaystyle\lim_{\overrightarrow{i}}\dim_{\F_p}{H_2(U_i\cap \widehat{G}_e,\F_p)}/[\widehat{G}_e:U_i\cap \widehat{G}_e]\leq \displaystyle\lim_{\overrightarrow{i}}{1/[\widehat{G}_e:\widehat{G}_e\cap U_i]}=0.$$ 
		Then by  Lemma \ref{l20} 
		\begin{align}
			\displaystyle\lim_{\overrightarrow{i}}{\dim_{\F_p}{H_2(U_i,\F_p)}/[\widehat{G}:U_i]}=0,
		\end{align}
		and hence by (2)
		\begin{align*}
			\lim_{\overrightarrow{i}}{\dim_{\F_p}{H_1(U_i,\F_p)\big)}/[\widehat{G}:U_i]}&=\lim_{\overrightarrow{i}}\dim_{\F_p}{L(U_i)/[\widehat{G}:U_i]} + \lim_{\overrightarrow{i}}\dim_{\F_p}{{H_2(U_i,\F_p)\big)}/[\widehat{G}:U_i]}\\
			&=-\chi_p(\widehat{G})+0=-\chi_p(\widehat{G}). 
	\end{align*}
By Theorem \ref{TZa}(b),		
	\begin{align*}
	\lim_{\overrightarrow{i}}{\dim_{\F_p}{H_1(U_i,\F_p)\big)}/[\widehat{G}:U_i]}=-\chi_p(\widehat{G})=-\chi(G)\geq 0. 
\end{align*}
		
	\end{proof}

The following fact will  be useful for  showing the analogous of  Bridson and  Howie's Theorem  for the case of the profinite completion  of limit groups.
	\begin{prop}\label{p:4}
		Let $G$ be a torsion free profinite group  and let $N$ be  a normal subgroup of $G$.  Let   $p$ be  a prime number and suppose that there is   a  sequence $\{V_{i}\}_{i\geq 1}$ of open normal  subgroups    of ${G}$  such that   $V_{i+1}\leq V_i$ for all  $i\geq 1$ and  $\displaystyle\bigcap_{i\geq 1}{V_i}=1$  satisfying the following conditions
		\begin{itemize}
			\item[$(i)$] There is $k_1>0$, such that  $\dim_{\F_p}H_1(NV_i,\F_p)\leq k_1[G: NV_i]$, for all $i\geq 1$. 
			\item[$(ii)$] There is  $k_2>0$, such that  $\dim_{\F_p}H_1(N\cap V_i,\F_p)\leq k_2[N:N\cap V_i]$, for all $i\geq 1$.
		\end{itemize}  If $\displaystyle\lim_{i\rightarrow\infty}    \displaystyle{\dfrac{\dim_{\F_p}H_1(V_i,{\F_p})}{[G:V_i]}}$ exists and  is  positive, then $|G:N|<~\infty$.
	\end{prop}
	
	\begin{proof}
		Suppose by contradiction that  $[G:N]=\infty$. Let $\epsilon>0$ arbitrary. Then $|N|=[G:N]=\infty$    implies that  exists $i_0\in \N$ and  $V_{i_0}\unlhd_\circ G$ such that 
		\begin{align}\label{eq1.000}
			|G:NV_{i_0}|,|NV_{i_0}:V_{i_0}|>\dfrac{k_2+k_1}{\epsilon}
		\end{align}
		Let  $H_1(V_i,\F_p)\neq 0$ for $i\geq i_0$. By the  Five Term  Exact Sequence  (see  \cite[\S 7.2, Thm.~7.2.6]{ZR}) we have the followings exact sequences 
		\begin{align}\label{e20}
			H_1(N\cap V_i,\F_p)_{V_i/V_i\cap N}\xrightarrow{\alpha_1} H_1(V_i,\F_p)\xrightarrow{\alpha_2} H_1(V_i/V_i\cap N,\F_p)\rightarrow 0 
		\end{align}
	and 
	\begin{align}\label{e21}
		H_1(N,\F_p)_{NV_i/N}\xrightarrow{\alpha_1} H_1(NV_i,\F_p)\xrightarrow{\alpha_2} H_1(NV_i/N,\F_p)\rightarrow 0 
	\end{align}
By   (\ref{e20}) and (ii), we have  that 
		\begin{align}\label{eq21}
			\dim_{\F_p}(\ker(\alpha_2))\leq  \dim_{\F_p}H_1(N\cap V_i,\F_p)\leq k_2[N:N\cap V_i].
		\end{align} 
		Furthermore by  (\ref{e21}) and (i), 
			\begin{align}\label{eq20}
			\dim_{\F_p}H_1(V_i/V_i\cap N,\F_p)\leq \dim_{\F_p} H_1(NV_i,\F_p)\leq  k_1[G: NV_i].
		\end{align}
		It follows  by (\ref{e20}), (\ref{eq21})  and  (\ref{eq20}) 
		\begin{align}\label{eq40}
			\nonumber\dim_{\F_p}H_1(V_i,\F_p)&=\dim_{\F_p}\image(\alpha_2)+ \dim_{\F_p}\ker(\alpha_2)\\  \nonumber &=\dim_{\F_p}H_1(V_i/V_i\cap N,\F_p)+ \dim_{\F_p}\ker(\alpha_2) \\  
			&\leq \dfrac{k_1|G:V_i|}{|NV_i:V_i|}+ \dfrac{k_2|G:V_i|}{|G:NV_i|}
		\end{align}
		Hence by  (\ref{eq1.000}) and (\ref{eq40}),
		\begin{align}
			\nonumber\dfrac{	\dim_{\F_p}H_1(V_i,\F_p)}{|G:V_i|}&\leq \dfrac{k_1}{|NV_i:V_i|}+ \dfrac{k_2}{|G:NV_i|}\\  \nonumber
			&\leq\dfrac{k_1+k_2}{\min\{|G:NV_i|,|NV_i:V_i|\}}\\
			&<\epsilon
		\end{align}
		Therefore, 
		by  $i\geq i_0$ we have 
		\begin{align}\label{eq50}
			0\leq \dfrac{	\dim_{\F_p}H_1(V_i,\F_p)}{|G:V_i|}< \epsilon
		\end{align}
		Hence  by (\ref{eq50} )  $\displaystyle\lim_{i\rightarrow\infty}   \dfrac{	\dim_{\F_p}H_1(V_i,\F_p)}{|G:V_i|}=0$ which is a contradiction.

	\end{proof}
	As a consequence one concludes the Theorem B. 
	
	\begin{proof}[\textbf{Proof of Theorem B} ]
 We  havet that 	$\widehat{G}$ is  torsion free (see \cite[Thm.~A and Thm.~E ]{TZ}) and $\chi_p(\widehat{G})=\chi(G)<0$ (see  Theorem \ref{TZa}(b)).  Note that, exists   a  sequence $\{V_{i}\}_{i\geq 1}$ of open normal  subgroups    of ${G}$  such that   $V_{i+1}\leq V_i$ for all  $i\geq 1$ and  $\displaystyle\bigcap_{i\geq 1}{V_i}=1$, hence, by    Theorem~A(3) we have $\displaystyle\lim_{i\rightarrow\infty} \displaystyle{\dfrac{\dim_{\F_p}H_1(V_i,{\F_p})}{[G:V_i]}}=-\chi_p(G)>0$. Note that  $$\dim_{\F_p}H_1(NV_i,\F_p)\leq d(NV_i)\leq d(G)[G: NV_i]$$ and   $$\dim_{\F_p}H_1(N\cap V_i,\F_p)\leq d(N\cap V_i)\leq d(N)[N:N\cap V_i].$$ Therefore    the Proposition  \ref{p:4}  yields the claim. 
	\end{proof}                                                     
	In \cite{AD}, D.~Kochloukova and P.~Zalesskii
		answered the analogous of Bridson and  Howie's Theorem for pro-$p$ limit groups (see \cite[\S 6, Thm.~6.7]{AD}). The Proposition~\ref{p:4} gives us a different proof of this claim.             
		 \begin{cor}
			Let  $G$ be a  pro-p limit group  non-abelian  and  let   $N$ be a finitely generated  normal closed subgroup  of  $G$. Then  $[{G}:N]<\infty$.  
		\end{cor}
		\begin{proof}
			The  Euler characteristic $\chi(G)$ of $G$ is negative (see \cite[\S 3, Thm.~3.6]{ABI}) and $G$ is  torsion free  (see \cite[\S 3, Thm.~3.5(ii)]{AD}).  Note that, there exists   a  sequence $\{V_{i}\}_{i\geq 1}$ of open normal  subgroups    of ${G}$  such that   $V_{i+1}\leq V_i$ for all  $i\geq 1$ and  $\displaystyle\bigcap_{i\geq 1}{V_i}=1$. Since $\displaystyle\lim_{i\rightarrow\infty} \displaystyle{\dfrac{\dim_{\F_p}H_1(V_i,{\F_p})}{[G:V_i]}}=-\chi(G)>0$ (see  \cite[\S 5, Thm.~5.3.(iii)]{AC}),  Proposition~\ref{p:4}  yields the claim. 
		\end{proof}
	 
Let $G=F_2\star (\Z\times \Z)$,  where $F_2$  is a free group of  rank two. We have that $G$ is a non-abelian limit group, and  by Kurosh'~Theorem,  each abelian subgroup of $G$ has  rank less than or equal to 2. 
We calculate    the $p$-deficiency  of the profinite completion of a special kind of limit group.

\begin{thm}\label{thmD}
	Let  $G$ be a non-abelian limit group   such that each abelian subgroup has  rank less than or equal to 2 and let $p$ be a prime number. Then  $\de_p(\widehat{G})\geq 2$.
\end{thm}
\begin{proof}
		We proceed by induction on $n=\het(G)$ (see \S\ref{ss:hgt}).
		If $n=0$, then $G$ is a finitely generated free pro-$p$ group satisfying $\de(G)\geq 2$, hence  by (\ref{eq001})
		\begin{equation}
		\label{eq:mt1}
		\de_p(\widehat{G})\geq \de(\widehat{G})\geq \de(G)\geq 2. 
		\end{equation} 
		and  the claim follows. 
		So assume that $G$ is a limit group of height $\het(G)=n\geq 1$, and that the claim holds
		for all limit groups of height less or equal to $n-1$. By  Proposition~\ref{prop:prin},
		$G$ is isomorphic to the fundamental group $\pi_1(\Upsilon,\Lambda,\caT)$
		of a graph of groups $\Upsilon$ based on a finite connected graph $\Lambda$ whose edge groups are either infinite cyclic or trivial, and whose vertex groups are either limit group of height at most $n-1$ or free abelian groups. 
		Applying induction on $s(\Lambda)=|V(\Lambda)|+|E(\Lambda)|$ it suffices to consider
		the following two cases:
		
		\begin{itemize}
			\item[(I)] $G=G_1 \star_{C} G_2$ and $G_i$ is either a limit group of height
			at most $n-1$ or abelian, and $C$ is either infinite cyclic or trivial, $i\in\{1,2\}$;
			\item[(II)] $G=\HNN_\phi(G_1,C,t)$ where $G_1$ is either a limit group of height
			at most $n-1$ or abelian, and $C$ is either infinite cyclic or trivial.
		\end{itemize}

		\noindent
		{\bf Case I:} Let  $G=G_1\star_CG_2$. We distinguish two cases. 
		
		\noindent
		{\bf (1)} $C\neq 1$. Then either $G_1$ or $G_2$ is non-abelian. Otherwise, one would conclude that 
		$\chi(G)= \chi(G_1)+\chi(G_2)-\chi(C)=0$ and $G$ must be abelian
		(see Lemma~\ref{l4}) which was excluded by hypothesis.  Since the profinite topology on $G$ is efficient (see  \cite[\S 3, Thm.~3.8]{gru} ), one has $\widehat{G}=\widehat{G}_1\amalg_{\widehat{C}}\widehat{G}_2$.   Hence without loss of generality we may assume that $G_1$ is non-abelian, and,
		by induction, $\de_p(\widehat{G}_1)\geq 2$. If $G_2$ is also non-abelian, then, 	by induction, $\de_p(\widehat{G}_2)\geq 2$. Otherwise, if $G_2$ is abelian, then, by hypothesis we have    
		\begin{align}
			\de_p(\widehat{G}_2)\geq \de(\widehat{G}_2)\geq \de(G_2)\geq  1.
		\end{align}
		Let  $M$ a finite $\Z_p[[\widehat{G}]]$-module,  the Mayer-Vietoris sequence associated to
		$H^\bullet(\argu,M)$  gives  
		\begin{equation}
		\label{eq:baslem4}
		\xymatrix{
			0 \ar[r]	& H^0(\widehat{G},M)\ar[r]& 
			H^0(\widehat{G}_1,M)\oplus  H^0(\widehat{G}_2,M)\ar[r]&
			H^0(\widehat{C},M)\ar[d]^{\gamma_1}\\
			&H^1(\widehat{C},M)\ar[d]^{\gamma_2}& 
			H^1(\widehat{G}_1,M)\oplus H^1(\widehat{G}_2,M)\ar[l]&
			H^1(\widehat{G},M)\ar[l]\\
			& 	H^2(\widehat{G},M)\ar[r]&H^2(\widehat{G}_1,M)\oplus H^2(\widehat{G}_2,M)\ar[r]&H^2(\widehat{C},M)
		} 
		\end{equation} 
		Since  $H^2(\widehat{C},M)=0$, by  (\ref{eq:baslem4}) counting dimension we have
  		\begin{align}\label{eq101}
			-\dim H^2(\widehat{G},M)=&-\dim H^2(\widehat{G}_1,M)-\dim H^2(\widehat{G}_2,M)+	\dim H^2(\widehat{C},M)\\ &\nonumber -\dim H^1(\widehat{G},M) +
			\dim H^1(\widehat{G}_1,M)+\dim H^1(\widehat{G}_2,M)\\ &\nonumber -\dim H^1(\widehat{C},M)+\dim H^0(\widehat{G},M)-\dim H^0(\widehat{G}_1,M)\\ &\nonumber -\dim H^0(\widehat{G}_2,M)+\dim H^0(\widehat{C},M).
		\end{align}	
		As  $M$ is  a finite $\Z_p[[\widehat{G}]]$-module, then  $M$ is  a finite  $\Z_p[[\widehat{G}_i]]$-module for $i\in\{1,2\}$, and  $M$ is a  finite $\Z_p[[\widehat{C}]]$-module. Then, by  (\ref{eq01}) and (\ref{eq101}) we have 
		\begin{align}\label{et1}
				\overline{\chi_2}(\widehat{G},M)=	\overline{\chi_2}(\widehat{G}_1,M)+	\overline{\chi_2}(\widehat{G}_2,M)-	\overline{\chi_2}(\widehat{C},M).
		\end{align}
		Note that  
		\begin{align}\label{et2}
				\overline{\chi_2}(\widehat{G}_1,M)\geq \de_p(\widehat{G}_1)-1\geq 1,\,\,\,\,\, \text{and}\,\,\,\,\, 	\overline{\chi_2}(\widehat{G}_2,M)\geq \de_p(\widehat{G}_2)-1\geq 0.
		\end{align} 
		Furthermore, 
		\begin{align}\label{et3}
				\overline{\chi_2}(\widehat{C},M)=d(\widehat{C})-1=0\,\,\text{(see  \cite[\S 2, Example~2.6]{G})}.
		\end{align} Therefore, by (\ref{et1}), (\ref{et2}) and (\ref{et3})  we have  
		$\overline{\chi_2}(\widehat{G},M)\geq 1$, hence $\de_p(\widehat{G})\geq 2$.\\
		\noindent
		\\{\bf (2)}
		$C=1$. Then, by  (\ref{eq01}) and (\ref{eq101}),  
		\begin{equation}\label{e003}
			\overline{\chi_2}(\widehat{G},M)=	\overline{\chi_2}(\widehat{G}_1,M)+	\overline{\chi_2}(\widehat{G}_2,M)+1.
		\end{equation}
		If both $G_1$ and $G_2$   are abelian  then $\de(G_i)\geq 1$  for $i\in \{1,2\}$. Hence, we have 
		\begin{align}\label{ea1}
			1+	\overline{\chi_2}(\widehat{G}_i,M)\geq \de_p(\widehat{G}_i)\geq \de(\widehat{G}_i)\geq 1, \,\,\, \text{for}\,\, i\in \{1,2\}.
		\end{align}
		So, by (\ref{e003}) and  (\ref{ea1}) we have $	\overline{\chi_2}(\widehat{G},M)\geq 1$, hence $\de_p(\widehat{G})\geq 2$.  Then, without loss of generality we may assume that $G_1$ is non-abelian, and,
		by induction, $\de_p(\widehat{G}_1)\geq 2$, furthermore  $\de_p(\widehat{G}_2)\geq 1$.  It follows that  
		\begin{align}\label{et4}
				\overline{\chi_2}(\widehat{G}_1,M)\geq \de_p(\widehat{G}_1)-1\geq 1\,\,\,\,\,\, \text{and}\,\,\,\, 	\overline{\chi_2}(\widehat{G}_2,M)\geq \de_p(\widehat{G}_2)-1\geq 0.
		\end{align}
		Then, by (\ref{e003}) and (\ref{et4}) we have  $	\overline{\chi_2}(\widehat{G},M)\geq 2$, hence  $\de_p(\widehat{G})\geq 3$.
		
		\noindent\\
		{\bf Case II:} Let  $G=\HNN_\phi(G_1,C,t)=
		\langle\, G_1,t\mid\,t\,c\,t^{-1}=\phi(c)\,\rangle$ be an HNN-extension
		with   $C=\langle c \rangle$. If $C=1$, then $G=G_1\star  <t>$ is isomorphic to a free product.
		Hence the claim follows already from Case I. So we may assume that $C\not=1$.
		Note that $G_1$ must be non-abelian.
		Otherwise, one has $\chi(G)=\chi(G_1)-\chi(C)=0$, 
		and $G$ must be abelian (see  Lemma~\ref{l4}), a contradiction.	Since the profinite topology on $G$ is efficient (see  \cite[\S 3, Thm.~3.8]{gru} ), then  $\widehat{G}=~\HNN(\widehat{G}_1,\widehat{C},t)$. Hence, as  $G_1$ is non-abelian,   induction implies that $~\de_p(\widehat{G}_1)\geq2$. 
		
		Let  $M$ a finite $\Z_p[[\widehat{G}]]$-module,  the Mayer-Vietoris sequence associated to
		$H^\bullet(\argu,M)$  gives 
		\begin{equation}
		\label{eq:baslem44}
		\xymatrix{
			0 \ar[r]	& H^0(\widehat{G},M)\ar[r]& 
			H^0(\widehat{G}_1,M)\ar[r]&
			H^0(\widehat{C},M)\ar[d]^{\gamma_1}\\
			&H^1(\widehat{C},M)\ar[d]^{\gamma_2}& 
			H^1(\widehat{G}_1,M)\ar[l]&
			H^1(\widehat{G},M)\ar[l]\\
			& 	H^2(\widehat{G},M)\ar[r]&H^2(\widehat{G}_1,M)\ar[r]&H^2(\widehat{C},M)
		} 
		\end{equation} 
		Since  $H^2(\widehat{C},M)=0$, by  (\ref{eq:baslem44}) counting dimension we have   
		\begin{align}\label{eq1011}
				-\dim H^2(\widehat{G},M)=&-\dim H^2(\widehat{G}_1,M)+\dim H^2(\widehat{C},M) -\dim H^1(\widehat{G},M) \\ &\nonumber +
			\dim H^1(\widehat{G}_1,M) -\dim H^1(\widehat{C},M)+\dim H^0(\widehat{G},M) \\ &\nonumber -\dim H^0(\widehat{G}_1,M)+\dim H^0(\widehat{C},M).
		\end{align}	
		As  $M$ is  a finite $\Z_p[[\widehat{G}]]$-module, then  $M$ is  a finite  $\Z_p[[\widehat{G}_i]]$-module for $i\in\{1,2\}$, and  $M$ is a  finite $\Z_p[[\widehat{C}]]$-module. Then, by  (\ref{eq01}) and  (\ref{eq1011}) we have 
		\begin{align}\label{et11}
				\overline{\chi_2}(\widehat{G},M)=	\overline{\chi_2}(\widehat{G}_1,M)-	\overline{\chi_2}(\widehat{C},M).
		\end{align}
		Note that  
		\begin{align}\label{et22}
				\overline{\chi_2}(\widehat{G}_1,M)\geq \de_p(\widehat{G}_1)-1\geq 1,
		\end{align} 
		Therefore, by (\ref{et3}), (\ref{et11}) and (\ref{et22})   we have  
		$	\overline{\chi_2}(\widehat{G},M)\geq 1$, hence $\de_p(\widehat{G})\geq 2$.
		
	\end{proof}
	
	As a consequence one concludes the following.
	
	\begin{cor}\label{pr6.1}
		Let  $G$ be a non-abelian limit group   such that each abelian subgroup has  rank less than or equal to 2 and  let $N$ be   a  normal subgroup of $\widehat{G}$ such that $0<\dim_{\F_p}H^1(N,\F_p)<\infty$, for every prime number $p$ that divide $|N|$. Then  $~|\widehat{G}:N|<\infty$.
	\end{cor}
	
	\begin{proof} 
		Let $p$ be a prime number such that $p$ divides $|N|$. Then, from Theorem~\ref{thmD} we have $\de_p(\widehat{G})\geq 2$. Suppose by contradiction that   $|\widehat{G}:N|=\infty$, hence       $H^1(N,\F_p)$ is  infinite  (see Proposition \ref{pp1}) which is a contradiction. Therefore  $|\widehat{G}:N|<\infty$. 
	\end{proof}
	
\section{Homological approximations for a pro-$p$ limit group}
\begin{example}
The pro-p group $G=F\amalg_C F$, where  $F=F(x,y)$ is a free  pro-p group of rank 2 and $C$ is a self-centralized procyclic subgroup of $F$ generated by $x^p[x,y]$, is a pro-limit group whose abelianization has torsion. Indeed, the group  $F\amalg_C F$ is immersed in  $F\amalg_C A$, where  $A\simeq \Z_p^2$ with $A/C\simeq \langle a \rangle$, because  $F\amalg_C F\simeq F\amalg_C aFa^{-1}$  is a subgroup of $F\amalg_C A$  generated by $F$ and  $aFa^{-1}$. Thus,  $F\amalg_C F$ is a pro-p limit group. Furthermore  $G^{\ab}$  has an element of order $p$, because  $G^{\ab}=\langle x,y,z,w\,|\,(xz^{-1})^p=1,[x,y]=[x,z]=[x,w]=[y,z]=[y,w]=[z,w] \rangle$. 
	\end{example}

   A  non-trivial pro-p limit group   has infinite abelianization  (see \cite[\S 4, Cor. 4.5]{AD})  and from the previous example, it makes sense to study the \textit{p-rational rank } of a pro-p limit group. 
	
    For an finitely generated pro-p group $G$ the {\em p-rational rank} of $G$ is given by 
	\begin{equation}
	\rk_{\Q_p}(G):=\dim_{\Q_p}(\Q_p\otimes_{\Z_p}{G^{\ab})}=\dim_{\Q_p}\big(\Q_p\otimes_{\Z_p}{H_1(G,\Z_p)\big)}
	\end{equation}
	In particular, $\rk_{\Q_p}(G)\leq d(G^{\ab})\leq d(G)$. 
\begin{Lemma}\label{l1}
		Let  $G$ be  a pro-$p$ group of type $FP_\infty$ over  $\Z_p$. Then  
		\begin{align}
			\dim_{\Q_p}(\Q_p\otimes_{\Z_p}H_{j}(G,\Z_p)\big)\leq dim_{\F_p}H_j(G,\F_p)\,\,\,\text{ for all}\,\, j\geq0. 
		\end{align}
	\end{Lemma}
	\begin{proof}
		As  $G$ is of type $FP_{\infty}$ over $\Z_p$, then  $\dim_{\F_p}{H_j(G,\F_p)}$ is finite  and  $H_j(G,\Z_p)$ is a   finitely generated abelian pro-$p$  group,  for all $j\geq 0$.
		
		Consider the following exact short sequence 
		\begin{gather}\label{et100}
			0\rightarrow \Z_p\xrightarrow{p^*}\Z_p\rightarrow\F_p\rightarrow0,
			\intertext {where the map  $\Z_p\xrightarrow{p^*}\Z_p$ is the multiplication by $p$. Hence, by (\ref{et100}) we have the following long exact  sequence} \label{et101}\cdots\rightarrow H_j(G,\Z_p)\xrightarrow{p^*} H_j(G,\Z_p) \xrightarrow{\phi} H_j(G,\F_p)\rightarrow\cdots, \intertext {where the map}H_1(G,\Z_p)\xrightarrow{p^*} H_1(G,\Z_p)\,\,\,\,\, \text{is the multiplication by $p$.}\nonumber
		\end{gather}
		By   \ref{et101}  we have
		\begin{align}\label{et103}	
			d\big(\frac{H_j(G,\Z_p)}{pH_j(G,\Z_p)}\big)\leq \dim_{\F_p}{H_j(G,\F_p)}. 
		\end{align}
		Note that 	
		\begin{align}
			\label{eq31}
			d\big(\dfrac{H_{j}(G,\Z_p)}{Tor(H_{j}(G,\Z_p))}\big)\leq d\big(\dfrac{H_j(G,\Z_p)}{pH_j(G,\Z_p)}\big)\intertext{where $Tor(H_{j}(G,\Z_p))$ is the torsion group of $H_{j}(G,\Z_p) $. As} 		
			\dim_{\Q_p}(\Q_p\otimes_{\Z_p}H_{j}(G,\Z_p)\big)=d\big(\frac{H_{j}(G,\Z_p)}{Tor(H_{j}(G,\Z_p))}\big),		
		\end{align}
		then by (\ref{et103}) and (\ref{eq31})   
		we get the result.
		
	\end{proof}
	Now,  we show the analogous of D.~Koch\-loukova and P.~Zalesskii's Theorem (see \cite[\S 5, Thm. 5.3]{AC})     for the field  $\Q_p$.  
	\begin{proof}[\textbf{Proof of Theorem C} ]
 \noindent
				\\ \item[\textbf{(1).}] As  $G$ is of  type  $FP_\infty$ over $\Z_p$ (see  \cite[\S 4, Cor. 4.4]{AC}), then     $U_i$ is of type  $FP_\infty$ over $\Z_p$,  for all $i\geq 1$. By Lemma \ref{l1}, we have for $j\geq 3$
			\begin{align}\label{et107}	
				\dim_{\Q_p}(\Q_p\otimes_{\Z_p}H_{j}(U_i,\Z_p)\big)/[G:U_i]\leq \dim_{\F_p}H_j(U_i,\F_p)/[G:U_i]. 
			\end{align}
			Now, by Theorem  5.3(i) in \cite{AC}  we have for $j\geq 3$
			\begin{align}\label{et108}
				\displaystyle\lim_{\overrightarrow{i}}{\dim_{\F_p}H_j(U_i,\F_p)/[G:U_i]} =0.
			\end{align}
			It follows  by (\ref{et107}) and (\ref{et108})
			$$\displaystyle\lim_{\overrightarrow{i}}{\dim_{\Q_p}(\Q_p\otimes_{\Z_p}H_{j}(U_i,\Z_p)\big)/[G:U_i]} =0$$ for  $j\geq 3$.
			
			\item[\textbf{(2).}] As $G$ is of type $FP_{\infty}$ over $\Z_p$, then  the p-characteristic of Euler is equivalent to the following expression 
			\begin{gather}
				\nonumber\chi_p(G)=\displaystyle\sum_{0\leq j\leq cd_p(G)}{(-1)^j \dim_{\F_p}{H_j(G,\F_p)}}\\\label{et109}=
				\displaystyle\sum_{0\leq j\leq cd_p(G)}{(-1)^j \dim_{\Q_p}{\big(\Q_p\otimes_{\Z_p}H_j(G,\Z_p)\big)}}
				\\\intertext{Since   $\chi_p(G)=\chi_p(U_i)/[G:U_i]$,  then by (\ref{et109})  we have  for all  $i\geq 1$}
				\label{et110}\chi(G)=\displaystyle\sum_{0\leq j\leq cd_p(G)}{(-1)^j \dim_{\Q_p}{\big(\Q_p\otimes_{\Z_p}H_j(U_i,\Z_p)\big)/[G:U_i]}}
			\end{gather}
			We define  a map   $L(U):=\dim_{\Q_p}{(\Q_p\otimes_{\Z_p}H_1(U,\Z_p))}-\dim_{\Q_p} (\Q_p\otimes_{\Z_p}H_2(U,\Z_p))$,  where $U\displaystyle\leq_o{G}$. Hence, by (\ref{et110}) we have for $i\geq 1$  
			\begin{align}\label{et111}	
				\chi(G)=\displaystyle\sum_{3\leq j\leq cd_p(G)}{(-1)^j \dim_{\Q_p}{\big(\Q_p\otimes_{\Z_p}H_j(U_i,\Z_p)\big)/[G:U_i]}}-L(U_i)/[G:U_i]+1/[G:U_i]
			\end{align} 
			It follows by (\ref{et111}) that  
			\begin{align}\label{et112}	
				\chi(G)=\displaystyle\sum_{3\leq j\leq cd_p(G)}{(-1)^j \displaystyle\lim_{\overrightarrow{i}}\dim_{\Q_p}{\big(\Q_p\otimes_{\Z_p}H_j(U_i,\Z_p)\big)/[G:U_i]}}-\displaystyle\lim_{\overrightarrow{i}}L(U_i)/[G:U_i]+\displaystyle\lim_{\overrightarrow{i}}1/[G:U_i]
			\end{align} 
			Therefore, by (1)  and
			$$-\chi(G)=\displaystyle\lim_{\overrightarrow{i}}{L(U_i)/[G:U_i]}$$
			
			\item[\textbf{(3).}]  By  Lemma \ref{l1}, we have   
			\begin{align}\label{et113}
				\dim_{\Q_p}(\Q_p\otimes_{\Z_p}H_{2}(U_i,\Z_p)\big)/[G:U_i]\leq dim_{\F_p}H_2(U_i,\F_p)/[G:U_i]
			\end{align}
			Now, by Theorem  5.3(ii) in \cite{AC} we have for $j\geq 3$
			\begin{align}\label{et114}
				\displaystyle\lim_{\overrightarrow{i}}{\dim_{\F_p}H_2(U_i,\F_p)/[G:U_i]} =0.
			\end{align}
			It follows  by (\ref{et113}) and (\ref{et114}) 
			\begin{align}\label{ec1}
				\displaystyle\lim_{\overrightarrow{i}}{\dim_{\Q_p}(\Q_p\otimes_{\Z_p}H_{2}(U_i,\Z_p)\big)/[G:U_i]} =0
			\end{align}
			Note that 
			\begin{align}
				\displaystyle\lim_{\overrightarrow{i}}{\dim_{\Q_p}{\big(\Q_p\otimes_{\Z_p}H_1(U_i,\Z_p)\big)}/[G:U_i]}=\lim_{\overrightarrow{i}}{L(U_i)/[G:U_i]} +\nonumber\\ \displaystyle\lim_{\overrightarrow{i}}{\dim_{\Q_p}(\Q_p\otimes_{\Z_p}H_{2}(U_i,\Z_p)\big)/[G:U_i]}
			\end{align}
			Hence, by (2) and  (\ref{ec1}) we have  
			$$\displaystyle\lim_{\overrightarrow{i}}{\dim_{\Q_p}{\big(\Q_p\otimes_{\Z_p}H_1(U_i,\Z_p)\big)}/[G:U_i]}=-\chi_p(G)$$.
		
	\end{proof}
	As consequence  we have the  Corollary D. 
	\begin{proof}[\textbf{Proof of Corollary D}]
		Note that 
		\begin{align}\label{et115}
			d(U_i)\geq d(U_i^{\ab})=\rk_{\Q_p}(U_i)+T(U_i) 
		\end{align}
		Now, from Theorem  5.3(iii) in \cite{AC} we have 
		\begin{align}\label{et116}
			\displaystyle\lim_{\overrightarrow{i}}{\dim_{\F_p}H_1(U_i,\F_p)/[G:U_i]} =-\chi_p(G).
		\end{align}
		Furthermore, by Theorem  C(3) we have 
		\begin{align}\label{et117}
			\displaystyle\lim_{\overrightarrow{i}}{{\rk_{\Q_p}(U_i)}/[G:U_i]}=-\chi_p(G) 
		\end{align}
		Hence, by (\ref{et115}), (\ref{et116}) and (\ref{et117}) we have 
		\begin{align}
			\displaystyle\lim_{\overrightarrow{i}}{T(U_i)}/[G:U_i]=0.
		\end{align}
	\end{proof}
The following Lemma is the  pro-p version  of  Lemma 3.1 in \cite{jh}. 
\begin{Lemma}
	\label{l2}
	Let $G_1$ and  $G_2$ be  finitely generated pro-p groups, and let $C=\langle c \rangle$ be an  procyclic group isomorphic to $\Z_p$ or the trivial group.
	\begin{itemize}
		\item[(a)] If $G=G_1\amalg_{C}G_2$ is a free pro-p  product with amalgamation in $C$, then 
		\begin{equation}
		\label{eq:baslem1}
		\rk_{\Q_p}(G)=\rk_{\Q_p}(G_1)+\rk_{\Q_p}(G_2)-\rho(G),
		\end{equation}
		where $\rho(G) \in \{0,1\}$. Moreover, if $C=1$, then $\rho(G)=0$. 
		\item[(b)] If $G=\HNN_\phi(G_1,C,t)=\langle\, G_1,t\mid t\,c\, t^{-1}=\phi(c)\,\rangle$ is an pro-p HNN-extension with 
		equalization in $C\subseteq G_1$, then 
		\begin{equation}
		\label{eq:baslem2}
		\rk_{\Q_p}(G)=\rk_{\Q_p}(G_1)+\rho(G),
		\end{equation} 
		where $\rho(G)\in\{0,1\}$. One has an exact sequence 
		\begin{equation}
		\label{eq:baslem3}
		\xymatrix{
			C\otimes_{\Z_p}\Q_p\ar[r]^-{\alpha}& 
			G_1^{\ab}\otimes_{\Z_p}\Q_p \ar[r]^-{\beta}&
			G^{\ab}\otimes_{\Z_p}\Q_p  \ar[r]&
			\Q_p\ar[r]&0,
		}
		\end{equation}
		Moreover, 
		\begin{itemize}
			\item[(1)] $\rho(G)=0$ if, and only if, $\alpha$ is injective;
			\item[(2)] $\rho(G)=1$ if, and only if, $\alpha$ is the $0$-map. 
		\end{itemize}
	\end{itemize}
\end{Lemma}

\begin{proof}
	(a) Let $G=G_1\amalg_{C}G_2$. 
	The Mayer-Vietoris sequence associated to 
	$-\otimes_{\Z_p}\Q_p$ specializes to an exact sequence
	\begin{equation}
	\label{eq:baslem444}
	\xymatrix{
		&C^{ab}\otimes_{\Z_p} \Q_p \ar[r]^-{\alpha}& 
		(G_1^{ab}\otimes_{\Z_p} \Q_p)\oplus (G_2^{ab}\otimes_{\Z_p} \Q_p)\ar[r]^-{\beta}&
		G^{ab}\otimes_{\Z_p} \Q_p\ar[d]^\gamma\\
		0& \Q_p\ar[l]&\Q_p\oplus \Q_p\ar[l]&\Q_p\ar[l]
	}
	\end{equation} 
	In particular, $\gamma=0$, and this yields (a).
	
	\noindent
	(b) In this case the Mayer-Vietoris sequence specializes to
	\begin{equation}
	\label{eq:baslem555}
	\xymatrix{
		&C^{ab}\otimes_{\Z_p} \Q_p\ar[r]^-{\alpha}& 
		G_1^{ab}\otimes_{\Z_p} \Q_p\ar[r]^-{\beta}&
		G^{ab}\otimes_{\Z_p} \Q_p\ar[d]\\
		0& \Q_p\ar[l]&\Q_p\ar[l]&\Q_p\ar[l]_-{\delta}
	}
	\end{equation} 
	In particular, $\delta=0$ which yields \eqref{eq:baslem3}, and thus also \eqref{eq:baslem2}. The final
	remarks (1) and (2) follow from the fact that
	$\dim(\image(\alpha))\in\{0,1\}$, and that $\dim(\image(\alpha))=1$ if, and only if, $\alpha$ is injective.
\end{proof}

From Lemma \ref{l2} one concludes the following Proposition. 

\begin{prop}\label{l12}
	Let  G be a   non-procyclic pro-p limit group, then   $$\dim_{\Q_p}(\Q_p\otimes_{\Z_p}G^{\ab})\geq  2.$$	 
\end{prop}

\begin{proof}
	
	We proceed by induction on  the height  $n=\het(G)$ of $G$.
	If $n=0$, then $G$ is a finitely generated abelian free pro-p  ou  free pro-p   group. By hypothesis we have  	\begin{equation}
	\label{eq:mt11}
	\rk_{\Q_p}(G)\geq 2
	\end{equation} 
	and hence the claim.
	So assume that $G$ is a non-procyclic pro-p limit group   of height $\het(G)=n\geq 1$, and that the claim holds
	for all non-procyclic pro-p limit groups  of height less or equal to $n-1$. By theorem 3.2 in ~\cite{ABI},  
	$G$ is isomorphic to the fundamental pro-p group $\pi_1(\Upsilon,\Lambda,\caT)$
	of a graph of pro-p groups $\Upsilon$ based on a finite connected graph $\Lambda$ which edge groups are either procyclic isomorphic to $\Z_p$  or trivial, and which vertex groups are either pro-p limit group of height at most $n-1$ or free abelian pro-groups. 
	Applying induction on $s(\Lambda)=|V(\Lambda)|+|E(\Lambda)|$ it suffices to consider
	the following two cases:
	
	\begin{itemize}
		\item[(I)] $G=G_1 \amalg_{C} G_2$ and $G_i$ is either a pro-p limit group  of height
		at most $n-1$ or abelian, and $C$ is either infinite procyclic isomorphic to $\Z_p$ or trivial, $i\in\{1,2\}$;
		\item[(II)] $G=\HNN_\phi(G_1,C,t)$ where $G_1$ is either a pro-p limit group of height
		at most $n-1$ or abelian, and $C$ is either infinite procyclic isomorphic to $\Z_p$ or trivial.
	\end{itemize}
	
	\noindent
	{\bf Case I:} Let  $G=G_1\amalg_CG_2$. We distinguish two cases.
	
	\noindent
	{\bf (1).}  If   $C$ is non-trivial, then either $G_1$ or $G_2$ is non-procyclic. Otherwise,  one would conclude that  $G$ must be procyclic  which was excluded by hypothesis. Hence without loss of generality we may assume that $G_1$ is non-procyclic. If $G_1$ is a  abelian pro-p group, then $\rk_{\Q_p}(G_1)\geq 2$.  If $G_1$ is a  non-abelian pro-p group, by induction, $\rk_{\Q_p}(G_1)\geq 2$. Furthermore, $\rk_{\Q_p}(G_2)\geq 1$. Hence,
	by applying Lemma \ref{eq:baslem1}(a), one concludes that 
	\begin{equation}
	\rk_{\Q_p}(G)=\rk_{\Q_p}(G_1)+
	\rk_{\Q_p}(G_2)-\rho(G)\geq 2+1-1=2.
	\end{equation}
	
	\noindent
	{\bf (2).} If $C=1$, by applying Lemma \ref{eq:baslem1}(a), one concludes that  
	\begin{align}
		\rk_{\Q_p}(G)=\rk_{\Q_p}(G_1)+\rk_{\Q_p}(G_2)\geq 1+1=2.	
	\end{align}
	{\bf Case II:} Let  $G=\HNN_\phi(G_1,C,t)=
	\langle\, G_1,t\mid\,t\,c\,t^{-1}=\phi(c)\,\rangle$ be an HNN-extension
	with   $C=\langle c \rangle$. If $C=1$, then $G=G_1\amalg <t>$ is isomorphic to a free pro-p  product.
	Hence the claim follows already from Case I. So we may assume that $C\not=1$. Note that $G_1$ must be non-procyclic.
	Otherwise, one has that  $G$ must be procyclic, a contradiction. 
	
	If $G_1$ is a  abelian pro-p group, then $\rk_{\Q_p}(G_1)\geq 2$.  If $G_1$ is a  non-abelian pro-p~group, by induction, $\rk_{\Q_p}(G_1)\geq 2$. Hence,
	by applying Lemma \ref{eq:baslem1}(b), one concludes that 
	\begin{equation}
	\rk_{\Q_p}(G)=\rk_{\Q_p}(G_1)+\rho(G)\geq 2+0=2.
	\end{equation}

\end{proof}

\begin{cor}
	Let $G$ a pro-p limit  group  and  $\rk_{\Q_p}(G)=1$, then   $G\simeq \Z_p$. 
\end{cor} 
	
	\bibliography{limitbib}

\providecommand{\bysame}{\leavevmode\hbox to3em{\hrulefill}\thinspace}
\providecommand{\MR}{\relax\ifhmode\unskip\space\fi MR }
\providecommand{\MRhref}[2]{%
  \href{http://www.ams.org/mathscinet-getitem?mr=#1}{#2}
}
\providecommand{\href}[2]{#2}
\begin{thebibliography}{10}

\bibitem{Bri}
{M.} Bridson and J.~Howie, \emph{Normalisers in limit groups}, Mathematische
  Annalen \textbf{337} (2007), no.~2, 385--394.

\bibitem{Brid}
M.~Bridson and D.~Kochloukova, \emph{Volume gradients and homology in towers of
  residually-free groups}, Mathematische Annalen \textbf{367} (2017), no.~3,
  1007--1045.

\bibitem{G}
F.~Grunewald, A.~Zapirain, A.~Pinto, and P.~Zalesskii, \emph{Normal subgroups
  of profinite groups of non-negative deficiency}, Journal of Pure and Applied
  Algebra \textbf{218} (2014), no.~5, 804 -- 828.

\bibitem{gru}
F.~Grunewald, A.~Zapirain, and P.~Zalesskii, \emph{Cohomological goodness and
  the profinite completion of bianchi groups}, Duke Math. J. \textbf{144}
  (2008), no.~1, 53--72.

\bibitem{jh}
J.~Gutierrez and T.~Weigel, \emph{Normal subgroups in limit groups of prime
  index}, Journal of Group Theory \textbf{21} (2017), no.~1, 83--100.

\bibitem{KMRS}
O.~G. Kharlampovich, A.~G. Myasnikov, V.~N. Remeslennikov, and D.~E. Serbin,
  \emph{Subgroups of fully residually free groups: algorithmic problems}, Group
  theory, statistics, and cryptography, Contemp. Math., vol. 360, Amer. Math.
  Soc., Providence, RI, 2004, pp.~63--101. \MR{2105437}

\bibitem{KM2}
O.~G. Kharlampovich and A.~M. Myasnikov, \emph{Irreducible affine varieties
  over a free group. {II}. {S}ystems in triangular quasi-quadratic form and
  description of residually free groups}, J. Algebra \textbf{200} (1998),
  no.~2, 517--570. \MR{1610664}

\bibitem{D}
D.~Kochloukova, \emph{On subdirect products of type {${\rm FP}_m$} of limit
  groups}, J. Group Theory \textbf{13} (2010), no.~1, 1--19. \MR{2604842}

\bibitem{AD}
D.~Kochloukova and P.~Zalesskii, \emph{On pro-p analogues of limit groups via
  extensions of centralizers}, Mathematische Zeitschrift \textbf{267} (2011),
  no.~1, 109--128.

\bibitem{AC}
\bysame, \emph{Subgroups and homology of extensions of centralizers of
  pro-{$p$} groups}, Math. Nachr. \textbf{288} (2015), no.~5-6, 604--618.
  \MR{3338916}

\bibitem{ZR}
L.~Ribes and P.~Zalesskii, \emph{Profinite groups}, Ergebnisse der Mathematik
  und ihrer Grenzgebiete : a series of modern surveys in mathematics, Springer,
  2000.

\bibitem{shu2}
M.~Shusterman, \emph{Ascending chains of finitely generated subgroups}, Journal
  of Algebra \textbf{471} (2017), 240 -- 250.

\bibitem{ABI}
I.~Snopche and P.~Zalesskii, \emph{Subgroup properties of pro-p extensions of
  centralizers}, Selecta Mathematica \textbf{20} (2014), no.~2, 465--489.

\bibitem{TZ}
P.~Zalesskii and T.~Zapata, arxiv.org/pdf/1711.01500.pdf, November 2017.

\end{thebibliography}
	\bibliographystyle{amsplain}
\end{document}